\documentclass{amsart}

\usepackage{amssymb}
\usepackage{latexsym}
\usepackage{amsmath}
\usepackage{euscript}
\usepackage{graphics}
\usepackage[active]{srcltx}

      \def\sF{{\mathfrak F}}
      
\def\sJ{{\mathfrak J}}      \def\sL{{\mathfrak L}}
      
\def\sP{{\mathfrak P}}      
\def\sS{{\mathfrak S}}   \def\sT{{\mathfrak T}}   \def\sU{{\mathfrak U}}
      
   \def\sZ{{\mathfrak Z}}

      \def\dC{{\mathbb C}}

      \def\dR{{\mathbb R}}

   \def\dZ{{\mathbb Z}}

\def\cP{{\mathcal P}}

\def\bm\chi{\mbox{\boldmath$\chi$}}

\def\diag{{\rm diag\,}}

\let\xker=\ker \def\ker{{\xker\,}}

\def\sg{\operatorname{sign}}

\def\sg{\operatorname{sign}}

\def\deg{\operatorname{deg}}

\newtheorem{thm}{Theorem}[section]

\newtheorem{prop}[thm]{Proposition}
\theoremstyle{definition}
\newtheorem{defn}[thm]{Definition}
\theoremstyle{remark}
\newtheorem{rem}[thm]{Remark}
\newtheorem{theorem}{Theorem}[section]

\theoremstyle{remark}

\newtheorem{remark}[theorem]{Remark}

\numberwithin{equation}{section}

\begin{document}

\title[Darboux transformations and polynomial mappings]
{On the relation between Darboux transformations and polynomial mappings}

\author{Maxim Derevyagin}
\address{
Maxim Derevyagin\\
Department of Mathematics MA 4-2\\
Technische Universit\"at Berlin\\
Strasse des 17. Juni 136\\
D-10623 Berlin\\
Germany}
\email{derevyagin.m@gmail.com}

\date{\today}


\begin{abstract} 
Let $d\mu$ be a probability measure on $[0,+\infty)$ such that its moments are finite. Then the Cauchy-Stieltjes transform $S$
of $d\mu$ is a Stieltjes function, which admits an expansion into a Stieltjes continued fraction. 
In the present paper, we consider a matrix interpretation of the unwrapping transformation $S(\lambda)\mapsto\lambda S(\lambda^2)$,
which is intimately related to the simplest case of polynomial mappings. More precisely, it is shown that this transformation is essentially
a Darboux transformation of the underlying Jacobi matrix. Moreover, in this scheme, the Chihara construction of solutions to the Carlitz problem 
appears as a shifted Darboux transformation.
\end{abstract}
\maketitle
%%%%%%%%%%%%%%%%%%%%%%%%%%%%%%%%%%%%%%%%%%%%%%%%%%%%%%%%%%%%%%%%%%%%%%%%%%%%%%%%%%

\section{Introduction}

It is well known that the monic Hermite polynomials $H_n$ can be expressed in terms of the monic Laguerre polynomials
$L_n^{(-1/2)}$ and $L_n^{(1/2)}$ by means of the formulas~\cite{Szego}
\[
 H_{2n}(\lambda)=L_n^{(-1/2)}(\lambda^2),\quad H_{2n+1}(\lambda)=\lambda L_n^{(1/2)}(\lambda^2),\quad
 n\in\dZ_+=\{0,1,2,\dots\}.
\]
It turns out that this construction, which unwraps the measure, can be generalized to the case of monic polynomials $P_j$ orthogonal with respect to a measure $
d\mu$ on $[0,+\infty)$~\cite{Car61},~\cite{DW63}. 
In other words, we can complete the system $\{P_j(\lambda^2)\}_{j=0}^{\infty}$ to a system of monic orthogonal polynomials $\Phi_j$ with the help of kernel polynomials $\widetilde{P}_j$,
i.e. the monic polynomials  
\[
\widetilde{P}_j(\lambda)=\frac{P_{j+1}(\lambda)-\frac{P_{j+1}(0)}{P_j(0)}P_{j}(\lambda)}{\lambda},
\]
which are orthogonal with respect to the measure $td\mu(t)$ on $[0,+\infty)$. Namely, the sequence $\Phi_j$ defined by the relation
\begin{equation}\label{ComplIntr}
 \Phi_{2j}(\lambda)=P_j(\lambda^2),\quad \Phi_{2j+1}(\lambda)=\lambda\widetilde{P}_j(\lambda^2),\quad j\in\dZ_+,
\end{equation}
is a sequence of orthogonal polynomials~\cite{Car61},~\cite{DW63}. Moreover, the measure of orthogonality for these polynomials is the following  measure
\[
 \frac{\sg t}{2}d\mu(t^2)
\]
on $\dR$~\cite{Chih64},~\cite{Si}.  This can be easily checked by recalling that the unwrapping transformation $P_j\mapsto\Phi_j$ allows to reduce a Stieltjes moment problem 
to a Hamburger one~\cite{Si},~\cite{Wall}. Actually, the transformation in question can be expressed in terms 
of the Cauchy-Stieltjes transform 
\[
 S(\lambda)=\int_{0}^{\infty}\frac{d\mu(t)}{t-\lambda}\in{\bf S}_{-\infty}
\]
of $d\mu$ as follows 
\[
 S(\lambda)\mapsto\lambda S(\lambda^2)=\int_{\dR}\frac{1}{t-\lambda}\frac{\sg t}{2}d\mu(t^2)\in{\bf N}_{-\infty}.
\]

Interestingly, iterations of the above simple construction can lead to orthogonal polynomials on several intervals. Moreover, fixed points of
the unwrapping transformation give rise to orthogonal polynomials on Cantor sets and have nontrivial applications in physics~\cite{BMM82}.
Furthermore, the idea of unwrapping measures generated a study of polynomial mappings~\cite{CI93},~\cite{CGI07},~\cite{GvanA88},~\cite{Peh03}.

Concerning the completion problem~\eqref{ComplIntr}, in~\cite{Chih64}, it is shown that it is not, in general, unique. More precisely,  if the measure $d\mu$ is supported on $[a,+\infty)$ 
then there are infinitely many ways to complete the system $\{P_j(\lambda^2)\}_{j=0}^{\infty}$ to a system 
of monic orthogonal polynomials. A construction of a family of completions to this problem was given in~\cite{Chih64}
by T. S. Chihara. The Chihara construction is the key ingredient to introduce and to study SDG maps~\cite{DVZ12},~\cite{DTVZh12}, which relate 
polynomials on the real line and the unit circle. 

The aim of this paper is twofold. The first goal is to give a matrix interpretation of the above constructions including the Chihara construction. 
The second purpose is to show that the generalized Jacobi matrices introduced and studied in~\cite{DD},~\cite{DD07},~\cite{De09},~\cite{DD10} 
implicitly appear in the theory of polynomial mappings. An introduction to the theory of generalized Jacobi matrices from the classical point of view is also given.

The paper is organized as follows. In Section~2, a preliminary information on the subclass ${\bf N}_{-\infty}$ of Nevanlinna functions is given.
In particular, the subclass ${\bf S}_{-\infty}(\subset{\bf N}_{-\infty})$ of Stieltjes functions is also discussed. The goal of the next section is to introduce
monic generalized Jacobi matrices associated with functions of the form $F(\lambda^2)$, where $F\in{\bf N}_{-\infty}$. By considering "symmetrized"
generalized Jacobi matrices, Section~4 demonstrates the difference between the generalized Jacobi matrices associated with $F(\lambda^2)$ and with $S(\lambda^2)$,
where $F\in{\bf N}_{-\infty}$ and $S\in{\bf S}_{-\infty}$. In Section~5, it is shown that the unwrapping transformation $S(\lambda)\mapsto\lambda S(\lambda^2)$,
$S\in{\bf S}_{-\infty}$, can be realized as a Darboux transformation. Finally, this matrix interpretation is extended in Section~6 to include the Chihara construction~\cite{Chih64}.

\section{The Stieltjes class and the unwrapping}
In this section we will give a brief description of the classical objects we are going to exploit in the present note. 

At first, recall that a Nevanlinna function is a 
function, holomorhic in the upper half-plane $\dC_+$, which maps $\dC_+$ into itself~\cite{Ach61},~\cite{KK74}. 
We will denote the class of all Nevanlinna functions by ${\mathbf N}$. Actually, in what follows we will only deal with Nevanlinna functions
 that admit the following asymptotic expansion
\begin{equation}\label{asymp}
F(\lambda)=-\frac{s_{0}}{\lambda}-\frac{s_{1}}{\lambda^2}
-\dots-\frac{s_{2n}}{\lambda^{2n+1}}+o\left(\frac{1}{\lambda^{2n+1}}\right),
\quad\lambda\widehat{\rightarrow }\infty,
\end{equation}
for all $n\in\dZ_+$, where  $s_j$, $j\in\dZ_+$, are some real numbers  (here $\lambda\widehat{\rightarrow }\infty$ means that $\lambda$
tends to $\infty$ nontangentially, i.e. inside the sector
$\varepsilon<\arg \lambda<\pi-\varepsilon$ for some $\varepsilon>0$).

As is known ~\cite{Ach61},~\cite{KK74},  all  Nevanlinna functions $F\in {\bf N}_{-\infty}$ 
can be represented as follows
\[
F(\lambda)=\int_{\dR}\frac{d\mu(t)}{t-\lambda},
\]
where $\mu$ is a bounded  measure supported on the real line $\dR$. Moreover, we have that
$s_j=\int_{\dR}t^jd\mu(t)$, $j\in\dZ_+$, are the moments of the measure $d\mu$~\cite{Ach61}.
Throughout this paper, we assume that {\it $d\mu$ is a nontrivial probability measure}, that is, 
$s_0=1$ and $\det(s_{j+k})_{j,k=0}^{n}>0$ for all $n\in\dZ_+$.

Once we have a measure with finite moments of all non-negative orders, it is possible to construct a sequence of orthogonal polynomials. It is well known that the sequence
$\{P_j\}_{j=0}^{\infty}$ of monic polynomials  orthogonal with
respect to $\mu$ satisfies a three-term recurrence
relation~\cite{Ach61},~\cite{JTh},~\cite{Wall}
\begin{equation}\label{recrelmonp}
\lambda P_j(\lambda) = P_{j+1}(\lambda) + b_jP_j(\lambda) + c_{j-1}P_{j-1}(\lambda),\quad j\in\dZ_+,
\end{equation}
with the initial conditions
\[
P_{-1}(\lambda)=0,\quad P_{0}(\lambda) =1,
\]
where $b_j\in\dR$ and $c_j>0$, $j\in\dZ_+$.
In fact, the relation~\eqref{recrelmonp} can be rewritten as follows
\begin{equation}\label{JacEigValPr}
Jp(\lambda)=\lambda p(\lambda),
\end{equation}
where  $p = (P_0, P_1, P_2,\dots)^{\top}$ and  $J$ is a semi-infinite tridiagonal matrix
of the form
\begin{equation}\label{monJac}
J=\begin{pmatrix}
{b}_{0}   & 1 &       &\\
{c}_{0}   &{b}_1    &{1}&\\
        &{c}_1    &{b}_{2} &\ddots\\
&       &\ddots &\ddots\\
\end{pmatrix}.
\end{equation}
The matrix $J$ is said to be {\it the monic Jacobi matrix associated with $\mu$} or, equivalently {\it with $F$}.
It should be stressed here that $F\in{\bf N}_{-\infty}$ uniquely determines $J$. However, $J$ can correspond to
infinitely many distinct functions of the class ${\bf N}_{-\infty}$~\cite{Ach61},~\cite{Wall}. In this case, by $F$ we simply understand 
a family of ${\bf N}_{-\infty}$-functions that generate the monic Jacobi matrix $J$.

There is another way to get to Jacobi matrices and three term recurrence relations. The starting point for this way is~\eqref{asymp}.
More precisely, the asymptotic expansion of $F\in{\bf N}_{-\infty}$ can be expanded into 
the following continued fraction~\cite{JTh},~\cite{Wall}
\begin{equation}\label{Jfraction}
F(\lambda)\sim-\frac{1}{\lambda-b_0-\displaystyle{\frac{c_0}
{\lambda-b_1-\displaystyle{\frac{c_1}{\ddots}}}}}=
-\frac{1}{\lambda-b_0}
\begin{array}{l} \\ -\end{array}
\frac{c_0}{\lambda-b_1}
\begin{array}{l} \\ -\end{array}
\frac{c_1}{\lambda-b_2}
\begin{array}{l}\\-\dots\end{array},
\end{equation}
which is called a $J$-fraction. At this point, it is clear how to construct the Jacobi matrix $J$~\eqref{monJac} from the $J$-fraction~\eqref{Jfraction}.
Besides, the polynomials $P_j$ appear in this scheme as denominator polynomials of the approximants to the $J$-fraction.
Potentially, this method is more convenient for generalizations since we can start with a generic asymptotic expansion (which is not
necessarily related to a positive measure) and try
to expand it into a good continued fraction (for instance, see~\cite{DD},~\cite{JTh},~\cite{KL79},~\cite{Wall}). 
Before going into details of such generalizations, let us  turn our attention to the Stieltjes functions which are basically the 
essence of the present study. 

We say that a function $S\in {\bf N}_{-\infty}$ is a Stieltjes function 
if it has the following integral representation
\begin{equation}\label{IntRofSFunction}
 S(\lambda)=\int_{0}^{\infty}\frac{d\mu(t)}{t-\lambda},
\end{equation}
i.e. the support of the measure is contained in or equal to $[0,\infty)$. The class of all such functions is denoted by ${\bf S}_{-\infty}$.

Evidently, every $S\in{\bf S}_{-\infty}$ has the following asymptotic expansion 
\begin{equation}\label{eq:asympS}
S(\lambda)-\frac{s_{0}}{\lambda}-\frac{s_{1}}{\lambda^2}
-\dots-\frac{s_{2n}}{\lambda^{2n+1}}+o\left(\frac{1}{\lambda^{2n+1}}\right),
\quad\lambda\widehat{\rightarrow }\infty,
\end{equation}
for all $n\in\dZ_+$ and the $J$-fraction expansion
\begin{equation}\label{JfracSt}
 S(\lambda)\sim-\frac{1}{\lambda-b_0}
\begin{array}{l} \\ -\end{array}
\frac{c_0}{\lambda-b_1}
\begin{array}{l} \\ -\end{array}
\frac{c_1}{\lambda-b_2}
\begin{array}{l} \\ -\end{array}
\frac{c_2}{\lambda-b_3}
\begin{array}{l}\\-\dots\end{array}.
\end{equation}
Remarkably, every Stieltjes function $S\in{\bf S}_{-\infty}$ can be also expanded into the Stieltjes fraction~\cite{JTh},~\cite{Wall}
\begin{equation}\label{StFraction1}
 S(\lambda)\sim
 -\frac{1}{\lambda} 
\begin{array}{l} \\ -\end{array}
\frac{d_1}{1}
\begin{array}{l} \\ -\end{array}
 \frac{d_2}{\lambda} 
\begin{array}{l} \\ -\end{array}
\frac{d_3}{1}
\begin{array}{l} \\ -\end{array}
 \frac{d_1}{\lambda} 
\begin{array}{l} \\ -\end{array}
\frac{d_1}{1}
\begin{array}{l}\\-\dots\end{array}
\end{equation}
which is shortly called an $S$-fraction.
The fact that all functions of the class ${\bf S}_{-\infty}$ generate two different continued fractions, could be puzzling (this is not true for the entire class ${\bf N}_{-\infty}$!).
However, the things become transparent as soon as one recalls another equivalent definition of the class 
${\bf S}_{-\infty}$: the function $F\in {\bf N}_{-\infty}$ belongs to ${\bf S}_{-\infty}$ if and only if
$\lambda F(\lambda^2)\in{\bf N}_{-\infty}$~\cite{KK74}. So, there are two functions $S(\lambda)$ and $\lambda S(\lambda^2)$ of
the class ${\bf N}_{-\infty}$ related to every function $S\in{\bf S}_{-\infty}$ and that is why $S$ has two continued fraction expansions. 
Indeed, if  $S$ is of the form~\eqref{IntRofSFunction} then one can easily check that
\begin{equation}\label{Smeasure}
 \lambda S(\lambda^2)=\int_{\dR}\frac{1}{t-\lambda}\frac{\sg t}{2}d\mu(t^2).
\end{equation}
Therefore, one can write
\[
 \lambda S(\lambda^2)\sim -\frac{1}{\lambda-\beta_0}
\begin{array}{l} \\ -\end{array}
\frac{\gamma_0}{\lambda-\beta_1}
\begin{array}{l} \\ -\end{array}
\frac{\gamma_1}{\lambda-\beta_2}
\begin{array}{l} \\ -\end{array}
\frac{\gamma_2}{\lambda-\beta_3}
\begin{array}{l}\\-\dots\end{array}.
\]
Further, we see from~\eqref{Smeasure} that the corresponding measure is symmetric. This fact implies that $\beta_j=0$, $j\in\dZ_+$ (see~\cite{Ach61},~\cite{Si},~\cite{Wall} for more details)
and thus we have
\[
 \lambda S(\lambda^2)\sim -\frac{1}{\lambda}
\begin{array}{l} \\ -\end{array}
\frac{\gamma_0}{\lambda}
\begin{array}{l} \\ -\end{array}
\frac{\gamma_1}{\lambda}
\begin{array}{l} \\ -\end{array}
\frac{\gamma_2}{\lambda}
\begin{array}{l} \\ -\end{array}
\frac{\gamma_3}{\lambda}
\begin{array}{l}\\-\dots\end{array}.
\]
Dividing the latter relation by $\lambda$ we get
\[
S(\lambda^2)\sim -\frac{1}{\lambda^2}
\begin{array}{l} \\ -\end{array}
\frac{\gamma_0}{1}
\begin{array}{l} \\ -\end{array}
\frac{\gamma_1}{\lambda^2}
\begin{array}{l} \\ -\end{array}
\frac{\gamma_2}{1}
\begin{array}{l} \\ -\end{array}
\frac{\gamma_3}{\lambda^2}
\begin{array}{l}\\-\dots\end{array},
\]
which becomes~\eqref{StFraction1} after the substitution $\lambda^2\to\lambda$. Actually, the transformation
\begin{equation}\label{UnwrapTr}
S(\lambda)\mapsto \lambda S(\lambda^2),
\end{equation}
which we will call the unwrapping by obvious reasons, is a way to reduce a Stieltjes moment problem
to a Hamburger moment problem~\cite[Section~87]{Wall} (see also~\cite[Theorem~2.13]{Si}). At the same time, this shows some importance of functions
of the form $S(\lambda^2)$ in the theory of $S$-fractions and $J$-fractions.

\section{Monic generalized Jacobi matrices}

Now we are in a position to discuss another class of asymptotic relations and the corresponding matrices. Namely, we are going 
to associate a block tridiagonal matrix with a function $\sF$ of the form
\begin{equation}\label{defFN}
 \sF(\lambda)=F(\lambda^2),
\end{equation}
where $F\in{\bf N}_{-\infty}$. Hence, the function $\sF$ admits the following asymptotic expansion
\begin{equation}\label{eq:asymp}
\sF(\lambda)=F(\lambda^2)\sim-\sum_{j=0}^\infty
\frac{s_j}{{\lambda}^{2j+2}},\quad\lambda\widehat{\rightarrow }\infty,
\end{equation}
where $\sim$ denotes the validity of~\eqref{asymp} for all $n\in\dZ_+$. Also, it will be useful to have all the powers in~\eqref{eq:asymp}
\begin{equation}\label{eq:asymp:defs}
\sF(\lambda)\sim-\sum_{j=0}^\infty
\frac{\mathfrak{s}_j}{{\lambda}^{j+1}},\quad\lambda\widehat{\rightarrow }\infty,
\end{equation}
where $\mathfrak{s_{2j}}=0$ and $\mathfrak{s}_{2j+1}=s_j$ for all $j\in\dZ_+$. As a matter of fact, the expansion
on the right hand side of~\eqref{eq:asymp:defs} can be easily expanded into a $P$-fraction (see~\cite{JTh} for the general definition of $P$-fractions).

\begin{prop}\label{GJM}
Let ${F}$ be a function of the class ${\bf N}_{-\infty}$. 
Then the function ${\mathfrak F}(\lambda)=F(\lambda^2)$ admits the expansion into the following continued fraction ($P$-fraction)
\begin{equation}\label{PfractionG}
-\frac{1}{\lambda^2-b_0}
\begin{array}{l} \\ -\end{array}
\frac{c_0}{\lambda^2-b_1}
\begin{array}{l} \\ -\end{array}
\frac{c_1}{\lambda^2-b_2}
\begin{array}{l}\\-\dots\end{array},
\end{equation}
where $c_j>0$ and $b_j\in\dR$ are the entries of the J-fraction~\eqref{Jfraction}.
 \end{prop}
 
 \begin{proof}
  The proof is straightforward from~\eqref{Jfraction} and~\eqref{defFN}.
 \end{proof}
Next, the basic theory of continued fractions~\cite{JTh} reads that  the denominators ${\mathfrak P}_j$ of the approximants to the fraction~\eqref{PfractionG} satisfy 
the three-term recurrence relation
\begin{equation}\label{defrecrelmonp}
\lambda^2 {\mathfrak P}_j(\lambda) =
{\mathfrak P}_{j+1}(\lambda) + b_j{\mathfrak P}_j(\lambda) + c_{j-1}{\mathfrak P}_{j-1}(\lambda),\quad j\in\dZ_+,
\end{equation}
with the initial conditions
\[
{\mathfrak P}_{-1}(\lambda)=0,\quad {\mathfrak P}_{0}(\lambda) =1,
\]
where $c_j>0$ and $b_j\in\dR$.
Obviously, $\deg {\mathfrak P}_j=2n$ and ${\mathfrak P}_j(\lambda)=P_j(\lambda^2)$,
where $P_j$ are monic orthogonal polynomials generated by $F$.

A natural desire coming from the classical setting is to represent~\eqref{defrecrelmonp} as an eigenvalue
problem for a matrix (i.e. to get a problem similar to~\eqref{JacEigValPr}). The main obstacle to do that is
the fact that~\eqref{defrecrelmonp} is not linear with respect to $\lambda$. So, let us linearize the 
system~\eqref{defrecrelmonp}. To this end, introduce a new family of polynomials:
\[
{\sT}_{2j}(\lambda)={\sP}_j(\lambda),\quad \sT_{2j+1}(\lambda)=\lambda\sP_j(\lambda),\quad
j\in\dZ_+.
\]
For the new sequence, \eqref{defrecrelmonp} takes the form
\begin{equation}\label{LinDef}
\begin{split}
\lambda\sT_{2j}(\lambda)=&\sT_{2j+1}(\lambda)\\
\lambda\sT_{2j+1}(\lambda)=&\sT_{2j+2}(\lambda)+b_j\sT_{2j}(\lambda)+c_{j-1}\sT_{2j-2}(\lambda).
\end{split}
\end{equation}
Now, the problem in question is linear and \eqref{defrecrelmonp} can be rewritten as follows
\[
\begin{pmatrix}
 \begin{tabular}{ l | c | r | c}
  $\begin{matrix}
  0& 1 \\
  b_0 & 0
 \end{matrix}$ & $\begin{matrix}
  0& 0 \\
  1 & 0
 \end{matrix}$ & $\begin{matrix}
  0& 0 \\
  0 & 0
 \end{matrix}$& $\dots$ \\
  \hline
  $\begin{matrix}
  0& 0 \\
  c_0 & 0
 \end{matrix}$  & $\begin{matrix}
  0& 1 \\
  b_1 & 0
 \end{matrix}$  & $\begin{matrix}
  0& 0 \\
  1 & 0
 \end{matrix}$& $\ddots$ \\
  \hline
  $\begin{matrix}
  0& 0 \\
  0 & 0
 \end{matrix}$ & $\begin{matrix}
  0& 0 \\
  c_1 & 0
 \end{matrix}$ & $\begin{matrix}
  0& 1 \\
  b_2 & 0
 \end{matrix}$& $\ddots$ \\
 \hline
 $\vdots$&$\ddots$&$\ddots$&$\ddots$
\end{tabular}
 \end{pmatrix}\begin{pmatrix}{\sT_0(\lambda)}\\{\sT_1(\lambda)}\\{\sT_2(\lambda)}\\{\sT_3(\lambda)}\\
 {\sT_4(\lambda)}\\{\sT_5(\lambda)}\\\vdots\end{pmatrix}=\lambda\begin{pmatrix}{\sT_0(\lambda)}
\\{\sT_1(\lambda)}\\{\sT_2(\lambda)}\\{\sT_3(\lambda)}\\
{\sT_4(\lambda)}\\{\sT_5(\lambda)}\\\vdots\end{pmatrix}.
\]
This brings us to the following concept.
\begin{defn}[\cite{DD},~\cite{DD10}, cf.~\cite{KL79}]\label{defGJM}
{\it A monic generalized Jacobi matrix} associated \linebreak
with the function ${\mathfrak F}$ is the tridiagonal block matrix
\begin{equation}\label{mJacobi}
\mathfrak{J}=\begin{pmatrix}
\mathfrak{B}_{0}   &\mathfrak{D}_{0}&       &\\
\mathfrak{C}_{0}   &\mathfrak{B}_1    &\mathfrak{D}_{1}&\\
        &\mathfrak{C}_1    &\mathfrak{B}_{2} &\ddots\\
&       &\ddots &\ddots\\
\end{pmatrix},
\end{equation}
where the entries are as follows 
\[
{\mathfrak B}_j=\begin{pmatrix} 0 & 1\\
              b_j & 0 \end{pmatrix}, 
\mathfrak{D}_{j}=
\begin{pmatrix}
0&0\\
1&0\\
\end{pmatrix},
\mathfrak{C}_{j}=
\begin{pmatrix}
0&0\\
{c}_j&0\\
\end{pmatrix}.
\]
\end{defn}

Actually, the monic generalized Jacobi matrix is constructed from the $P$-fraction \eqref{PfractionG}.
It turns out that one can associate a generalized Jacobi matrix of a similar type to any $P$-fraction (for more details see~\cite{De09,DD,DD07}).

Summing up we see that the $P$-fraction expansion of $\sF$ leads to the following eigenvector problem
\[
\sJ{\mathfrak t}(\lambda)=\lambda{\mathfrak t}(\lambda),
\]
where ${\mathfrak t}=(\sT_0,\sT_1,\dots)^{\top}$. Keeping in mind the classical case, one would expect some orthogonality 
of $\{\sT_j\}_{j=0}^{\infty}$ showing up.
%\subsection{Almost orthogonality} It is clear that the function $\sF$ admits the following asymptotic expansion
%\begin{equation}\label{asympFS}
%\sF(\lambda)\sim-\sum_{j=0}^\infty
%\frac{{\mathfrak s}_j}{{\lambda}^{j+1}}=
%-\sum_{j=0}^\infty
%\frac{{s}_j}{{\lambda}^{2j+2}},\quad\lambda\widehat{\rightarrow }\infty,
%\end{equation}
%where $s_j$ are the moments given by~\eqref{eq:asympS}, i.e. $s_j=\int_{0}^{\infty}t^jd\mu(t)$,
%where $d\mu$ is the measure corresponding to $S$. Unfortunately, as we have already seen, the moments
%${\mathfrak s}_j$ are not generally the moments of a probability measure or even a finite signed measure.
%Therefore, we have to 
Mimicking the classical case, let us consider a linear functional $\sS$ defined on the space $\cP$ of all polynomials in $\lambda$
in the following way
\[
\sS(\lambda^j)={\mathfrak s}_j,\quad j\in\dZ_+,
\]
where $\mathfrak{s}_j$ are defined by~\eqref{eq:asymp:defs}.
To begin with, observe that  $\sT_0$ is orthogonal to itself with respect to $\sS$, that is,
\begin{equation}\label{Ort1}
\sS(\sT_0(\lambda)\sT_0(\lambda))=\sS(1)={\mathfrak s}_0=0.
\end{equation}
This property demonstrates that $\sS$ is not a regular functional. In other words, $\det({\mathfrak s}_{i+j})_{i,j=0}^{k}$
can vanish for some $k\in\dZ_+$. Besides~\eqref{Ort1}, the system $\{\sT_j\}_{j=0}^{\infty}$
possesses the following orthogonality property.
\begin{thm}\label{AlmOrt}
We have that
\begin{equation}\label{Ort2}
\sS(\sT_{2j}(\lambda)\sT_k(\lambda))=\frac{D_j}{D_{j-1}}\delta_{2j+1 k}, \quad j,k\in\dZ_+,
\end{equation}
\begin{equation}\label{Ort3}
\sS(\sT_{2j+1}(\lambda)\sT_k(\lambda))=\frac{D_j}{D_{j-1}}\delta_{2j k}, \quad j,k\in\dZ_+,
\end{equation}
where $D_j=\det(s_{i+k})_{i,k=0}^{j}>0$, $D_{-1}=1$, $s_k$ are given by~\eqref{asymp}, and
$\delta_{i k}$ is the Kronecker delta.
\end{thm}
\begin{proof} Recall that the monic orthogonal polynomials corresponding to $F$ can be calculated
by the formula
\[
P_j(\lambda)=\frac{1}{D_{j-1}}\begin{vmatrix}
s_0&s_1&\dots&s_j\\
\hdotsfor{4}\\
s_{j-1}&s_j&\dots& s_{2j-1}\\
1&\lambda^1&\dots&\lambda^j\\
\end{vmatrix}, \quad j\in\dZ_+.
\]
We begin the proof of~\eqref{Ort2} and~\eqref{Ort3} by noticing that
\[
\sS(\sT_{2j}(\lambda)\lambda^{2l})=0,\quad l\in\dZ_+.
\]
The latter relation is true because $\sT_{2j}(\lambda)=P_j(\lambda^2)$ and $\sS(\lambda^{2r})={\mathfrak s}_{2r}=0$ for
$r=0,1,\dots,2l$. Thus, we have just proved~\eqref{Ort2} for $k=2l$ and~\eqref{Ort3} for $k=2l+1$.
The rest follows from the equality
\[
\sS(\sT_{2j}(\lambda)\lambda^{2l+1})=\frac{1}{D_{j-1}}
\begin{vmatrix}
s_0&s_1&\dots&s_j\\
\hdotsfor{4}\\
s_{j-1}&s_j&\dots& s_{2j-1}\\
s_l&s_{l+1}&\dots&s_{j+l}\\
\end{vmatrix}
\]
for $l=0,1,\dots j$.
\end{proof}

Evidently, one cannot say that the sequence $\{\sT_j\}_{j=0}^{\infty}$ is orthogonal in the conventional sense.
However, \eqref{Ort2} and~\eqref{Ort3} show that the family $\{\sT_j\}_{j=0}^{\infty}$ is bi-orthogonal to 
itself with respect to $\sS$. M.G.~Krein and H. Langer proposed to call such systems almost orthogonal in their paper about indefinite analogs of the Hamburger and 
Stieltjes moment problems~\cite{KL79} and some earlier work. At the same time, the polynomials $\sT_j$ resemble skew orthogonal polynomials~\cite{AFNvM00}
and inherit many of their algebraic properties.

\begin{remark}
 From the proof of Theorem~\ref{AlmOrt} we can also get the following relations
 \begin{equation}\label{indOrt1}
 \begin{split}
  \sS(\sT_{2j}(\lambda)\lambda^k)&=\frac{D_j}{D_{j-1}}\delta_{2j+1 k}, \quad k=0,1,\dots, 2j+1,\\
  \sS(\sT_{2j+1}(\lambda)\lambda^k)&=\frac{D_j}{D_{j-1}}\delta_{2j k}, \quad k=0,1,\dots, 2j.
  \end{split}
 \end{equation}
 It should be mentioned that the first formula in~\eqref{indOrt1} is a special case of~\cite[Proposition 3.2]{Peh03}.
\end{remark}

Particularly, the results of this section are valid for functions of the form $S(\lambda^2)$. To analyze these functions we will need the following concept. 

\begin{defn}[\cite{DD09},~\cite{J2000}]\label{deffunct+}
Let us say that a function ${\mathfrak F}$ is definitazable by the polynomial $p(\lambda)=\lambda$ if it admits the
asymptotic expansion
\begin{equation}\label{asympF}
{\mathfrak F}(\lambda)\sim-\frac{{\mathfrak
s}_{0}}{\lambda}-\frac{{\mathfrak
s}_{1}}{\lambda^{2}}-\dots-\frac{{\mathfrak s}_{2n}}{\lambda^{2n+1}}
-\dots,\quad\lambda\widehat{\rightarrow }\infty,
\end{equation}
with some ${\mathfrak s}_j\in\dR$, $j\in\dZ_+$, and
$\lambda{\mathfrak F}(\lambda)+{\mathfrak s}_{0}\in{\bf
N}_{-\infty}$.
\end{defn}
Definitazibility is a very important concept in the theory of operators in Krein spaces~\cite{J2000},~\cite{L1982}.
Generally, the concept means that a definitazable object, which acts in a Krein space, can be made  definite (related to a Hilbert space) 
by means of some transformations. It should be also mentioned that the definizability by polynomials was also used in~\cite{GvanA88} and in~\cite{Peh03}.

In view of~\eqref{Smeasure},  the function
\begin{equation}\label{DefSt}
 {\sZ}(\lambda):=S(\lambda^2),
\end{equation}
where $S\in{\bf S}_{-\infty}$, is a definitizable function. It could seem from~\eqref{Smeasure} that ${\sZ}$ is the Cauchy transform of a finite signed measure.
 This can be easily disproved by considering the simplest Stieltjes function $S_0(\lambda)=-1/\lambda$. Then 
 \[
  {\sZ}_0(\lambda)=S_0(\lambda^2)=-\frac{1}{\lambda^2},
 \]
which cannot be the Cauchy transform of a finite signed measure because it is the Cauchy transform of the first derivative $\delta'_{\{0\}}$ of the Dirac delta function. 
However, for example, if the support of the measure representing the original
function $S$ doesn't contain $0$ then ${\sZ}$ is the Cauchy transform of a signed measure.

\section{A brief tour of the self-adjoint case}

This section demonstrates the difference between the generalized Jacobi matrices associated with $\sF$ and with $\sZ$.
To accomplish that, we have to start by considering the symmetric Jacobi matrix 
\[
 J_s=\begin{pmatrix}
{b}_{0}   & \sqrt{{c}_{0}} &       &\\
\sqrt{{c}_{0}}   &{b}_1    &\sqrt{{c}_{1}}&\\
        &\sqrt{{c}_{1}}    &{b}_{2} &\ddots\\
&       &\ddots &\ddots\\
\end{pmatrix}
\]
corresponding to the function $F$.
As a matter of fact, to realize the principal contrast between the matrices generated by $\sF$ and by $\sZ$, it is enough
to see it for bounded operators. So, {\it in this section we can assume that $J_s$ is bounded in $\ell^2$}.
In other words, we restrict ourselves to the case of the measure $d\mu$ with a bounded support.

Clearly, one can repeat the construction of Section~2 in the symmetric case and that gives the generalized Jacobi matrix
\[
 \sJ_s=\begin{pmatrix}
 \begin{tabular}{ l | c | r | c}
  $\begin{matrix}
  0& 1 \\
  b_0 & 0
 \end{matrix}$ & $\begin{matrix}
  0& 0 \\
   \sqrt{{c}_{0}} & 0
 \end{matrix}$ & $\begin{matrix}
  0& 0 \\
  0 & 0
 \end{matrix}$& $\dots$ \\
  \hline
  $\begin{matrix}
  0& 0 \\
  \sqrt{{c}_{0}} & 0
 \end{matrix}$  & $\begin{matrix}
  0& 1 \\
  b_1 & 0
 \end{matrix}$  & $\begin{matrix}
  0& 0 \\
  \sqrt{{c}_{1}} & 0
 \end{matrix}$& $\ddots$ \\
  \hline
  $\begin{matrix}
  0& 0 \\
  0 & 0
 \end{matrix}$ & $\begin{matrix}
  0& 0 \\
  \sqrt{{c}_{1}} & 0
 \end{matrix}$ & $\begin{matrix}
  0& 1 \\
  b_2 & 0
 \end{matrix}$& $\ddots$ \\
 \hline
 $\vdots$&$\ddots$&$\ddots$&$\ddots$
\end{tabular}
 \end{pmatrix}.
\]
Evidently, $\sJ_s$ is not symmetric. However, the relations~\eqref{Ort2} and~\eqref{Ort3} suggest that we should consider the operator $\sJ_s$
in a different space. Namely, introduce the following inner product
\[
 [x,y]=(Gx,y),
\]
where $(\cdot,\cdot)$ is the usual $\ell^2$-inner product and the Gram matrix $G$ is defined as follows
\[
 G=\diag(G_0,G_1,G_2,\dots),\quad G_j=\begin{pmatrix}
                                       0&1\\
                                       1&0
                                      \end{pmatrix},
 \quad j\in\dZ_+.
\]
It should be stressed here that $[\cdot,\cdot]$ is not positive definite:
\[
 [e_0,e_0]=(Ge_0,e_0)=0, \quad e_0=(1,0,0,0,\dots)^{\top}.
\]
Though it is not pleasant to work with indefinite inner products, the choice is dictated by the fact that we want the multiplication operator on
$\cP$ provided with the inner product defined by $\sS$ to become the generalized Jacobi matrix $\sJ_s$ acting on the space $\ell^2(G)$ of $\ell^2$-vectors
equipped with  $[x,y]=(Gx,y)$ (see~\cite{DD},~\cite{KL79}). Notice the matrix $G$ satisfies the following conditions
\[
 G=G^*,\quad G^2=I.
\]
Thus, the space $\ell^2(G)$ is a Krein space (for example, see~\cite{L1982}). By definition, the norm of a Krein space
is the same as the norm of the Hilbert space generating the Krein space~\cite{L1982}. In particular, the norm of the Krein space $\ell^2(G)$ is 
the norm of the Hilbert space $\ell^2$.
\begin{prop}The generalized Jacobi matrix $\sJ_s$ associated with $\sF$ defined by~\eqref{defFN} is a bounded self-adjoint operator in
$\ell^2(G)$, that is
\[
 [\sJ_sx,y]=[x,\sJ_sy],\quad x,y\in\ell^2(G).
\]
Moreover, the spectrum $\sigma(\sJ_s)$ of $\sJ_s$ can be found by the formula
\begin{equation}\label{spform}
 \sigma(\sJ_s)=\{\lambda\in\dC: \lambda^2\in\sigma(J_s)\}\subset\dR\cup i\dR,
\end{equation}
 where  $\sigma(J_s)$ denotes the spectrum of $J_s$ acting in the Hilbert space $\ell^2$.
\end{prop}
\begin{proof}
The boundedness follows from the boundedness of $J_s$, which we assume in this section.
Next, one sees that the matrix
\[
 G\sJ_s=
 \begin{pmatrix}
 \begin{tabular}{ l | c | r | c}
  $\begin{matrix}
  b_0& 0 \\
  0 & 1
 \end{matrix}$ & $\begin{matrix}
   \sqrt{{c}_{0}}& 0 \\
  0 & 0
 \end{matrix}$ & $\begin{matrix}
  0& 0 \\
  0 & 0
 \end{matrix}$& $\dots$ \\
  \hline
  $\begin{matrix}
   \sqrt{{c}_{0}}& 0 \\
  0 & 0
 \end{matrix}$  & $\begin{matrix}
  b_1& 0 \\
  0 & 1
 \end{matrix}$  & $\begin{matrix}
  \sqrt{{c}_{1}}& 0 \\
   0& 0
 \end{matrix}$& $\ddots$ \\
  \hline
  $\begin{matrix}
  0& 0 \\
  0 & 0
 \end{matrix}$ & $\begin{matrix}
  \sqrt{{c}_{1}}& 0 \\
  0 & 0
 \end{matrix}$ & $\begin{matrix}
  b_2& 0 \\
  0 & 1
 \end{matrix}$& $\ddots$ \\
 \hline
 $\vdots$&$\ddots$&$\ddots$&$\ddots$
\end{tabular}
 \end{pmatrix}
\]
is symmetric and this proves the self-adjointness. Finally, formula~\eqref{spform} is true because $\sJ_s$ is
a linearization of the problem $(J_s-\lambda^2I)x=y$, $x,y\in\ell^2$ (see also~\cite[Theorem~5.3]{De09} for a more strict statement). 
\end{proof}

\begin{rem} For the bounded Jacobi matrix $J_s$, it is well known that its $m$-function $m_{J_s}(\lambda)=((J_s-\lambda I)^{-1}e_0,e_0)$
coincides with $F$~\cite{Ach61},~\cite{Si},~\cite{Wall}. It appears that the same holds for bounded generalized Jacobi matrices~\cite{De09}. Namely, applying~\cite[Proposition~2.8]{DD}
to $\sJ_s^{\top}$ we arrive at
\[
m_{\sJ_s}(\lambda)=[(\sJ_s^{\top}-\lambda I)^{-1}e_0,e_0]=\sF(\lambda)=F(\lambda^2), \quad \lambda\in\dC\setminus\sigma(\sJ_s). 
\]
\end{rem}

Actually, knowing that an operator is self-adjoint in a Krein space doesn't say much because the spectrum of a self-adjoint operator
in a Krein space can be fairly arbitrary and the structure of the operator can be rather wild~\cite{L1982}.
Therefore, it makes sense to study narrower classes of operators in Krein spaces. For example, one of such classes
is the class of definitizable operators. In fact, the spectral calculus for definitizable 
operators is constructed~\cite{L1982}.

\begin{defn}[\cite{L1982}] A bounded self-adjoint operator $A$ in the Krein space $\ell^2(G)$ is called
definitizable if there exists a polynomial $p$ such that
\[
 [p(A)x,x]\ge 0, \quad x\in\ell^2(G).
\]
\end{defn}
Now, we are in a position to formulate the following result.
\begin{thm}
The generalized Jacobi matrix $\sJ_s$ associated with $\sZ$ given by~\eqref{DefSt} is non-negative in $\ell^2(G)$, i.e.
\[
 [\sJ_sx,x]\ge 0, \quad x\in\ell^2(G),
\]
and, thus, $\sJ_s$ is definitizable by the polynomial $p(\lambda)=\lambda$. Furthermore, we have that
\[
 \sigma(\sJ_s)=\{\lambda\in\dC: \lambda^2\in\sigma(J_s)\}\subset\dR.
\]
\end{thm}
\begin{proof}
It is easy to see that 
\[
 G\sJ_s\cong \begin{pmatrix}
                J_c&0\\
                0&I
               \end{pmatrix},
\]
where $\cong$ denotes the unitary equivalence in the Hilbert space $\ell^2$.  Hence, the non-negativity of
$\sJ_s$ follows from the non-negativity of $J_s$ corresponding to the Stieltjes function (the statement about the non-negativity of
$J_s$ can be found in~\cite{Si}).
Then, it remains to apply~\eqref{spform} taking into account that $J_s$ is non-negative.
\end{proof}

\section{The unwrapping as a Darboux transformation}

The unwrapping transformation
\begin{equation}\label{UnwrTr1}
 {\bf S}_{-\infty}\ni S(\lambda)\mapsto\lambda S(\lambda^2)\in{\bf N}_{-\infty}
\end{equation}
plays an important role not only in mathematics~\cite{Si},~\cite{Wall} but also in physics~\cite{BMM82}. 
It also gave rise to the study of polynomial mappings~\cite{CI93},~\cite{CGI07},~\cite{GvanA88},~\cite{Peh03}.
The goal of  this section is to give a matrix interpretation of the transformation~\eqref{UnwrTr1}. Namely,
we split~\eqref{UnwrTr1} into two steps: 
\[
 {\bf S}_{-\infty}\ni S(\lambda)\mapsto S(\lambda^2)=\sZ(\lambda),\quad \sZ(\lambda)\mapsto\lambda S(\lambda^2)\in{\bf N}_{-\infty}.
\]
From the point of view of Jacobi matrices, the first step was considered in Section~2. Now we are going to interpret the second one. To be more precise, 
it will be shown that the second transformation is essentially a Darboux transformation (strictly speaking, a Christoffel transformation~\cite{SpZh95},~\cite{Zh97}) 
of the underlying monic generalized Jacobi matrix obtained after the first step.

We start by recalling the well-known facts about the Darboux transformations of monic Jacobi matrices~\cite{BM04}. 
Let $S$ belong to ${\bf S}_{-\infty}$, that is,
\[
 S(\lambda)=\int_{0}^{\infty}\frac{d\mu(t)}{t-\lambda}.
\] 
Then the corresponding monic orthogonal polynomials obey 
\begin{equation}\label{Nonvan}
 (-1)^jP_j(0)>0, \quad j\in\dZ_+.
\end{equation}
Indeed, we have that 
\[
 P_j(0)=\frac{1}{D_{j-1}}
 \begin{vmatrix}
s_0&s_1&\dots&s_j\\
\hdotsfor{4}\\
s_{j-1}&s_j&\dots& s_{2j-1}\\
1&0&\dots&0\\
\end{vmatrix}=\frac{(-1)^j}{D_{j-1}}
\begin{vmatrix}
s_1&\dots&s_j\\
\hdotsfor{3}\\
s_j&\dots& s_{2j-1}
\end{vmatrix}
\]
and $det(s_{1+l+k})_{l,k=0}^{j-1}$ are positive since the moments $\widetilde{s}_{k}=s_{k+1}$, $k\in\dZ_+$, are the moments of the nontrivial positive measure
$td\mu(t)$ on $[0,+\infty)$~\cite{Ach61},~\cite{Si},~\cite{Wall}. Consequently, the corresponding Jacobi matrix admits an $LU$-decomposition~\cite{BM04}. This means that 
$J$ can be represented as follows
\begin{equation}\label{clLUdec}
J=\begin{pmatrix}
{b}_{0}   & 1 &       &\\
{c}_{0}   &{b}_1    &{1}&\\
        &{c}_1    &{b}_{2} &\ddots\\
&       &\ddots &\ddots\\
\end{pmatrix}=
LU=\begin{pmatrix}
1   & 0 &       &\\
{l}_{1}   & 1    &0 &\\
        &{l}_2    &1 &\ddots\\
&       &\ddots &\ddots\\
\end{pmatrix}\begin{pmatrix}
{u}_{0}   & 1 &       &\\
{0}   &{u}_1    &{1}&\\
        &0    &{u}_{2} &\ddots\\
&       &\ddots &\ddots\\
\end{pmatrix},
\end{equation}
where the entries $u_n$ and $l_n$ can be found by the formulas
\begin{equation}\label{defLUform}
u_n=-\frac{P_{n+1}(0)}{P_n(0)}>0, \quad l_{n+1}=\frac{c_n}{u_{n}}>0,\quad n\in\dZ_+. 
\end{equation}

In fact, in this case, the factorization~\eqref{clLUdec} is a version of the famous Cholesky decomposition,
which is widely used in computational mathematics (for instance, see~\cite{Gau02}). 
Moreover, it turns out that a similar statetment is also true for monic generalized Jacobi matrices~\cite[Proposition~3.1]{DD10}
and the corresponding factorization is related to the Bunch-Kaufman decomposition~\cite{BK77}. 
In the case of the monic generalized Jacobi matrices associated with functions of the form~\eqref{DefSt}, 
the statement is rather obvious up to the matrix structure. 
\begin{prop}\label{ExLU}
Let $\sJ$ be a monic generalized Jacobi matrix associated
with $\sZ$ defined by~\eqref{DefSt}. Then $\sT_{2j}(0)=P_j(0)\ne 0$ for all $j\in\dZ_+$ and 
$\sJ$ admits the following factorization
\begin{equation}\label{ind_LU}
\mathfrak{J}=\mathfrak{L}\mathfrak{U},
\end{equation}
where $\mathfrak{L}$ and $\mathfrak{U}$ are block lower and upper
triangular matrices having the forms
\begin{equation}\label{eq:ind_LU}
\mathfrak{L}=\begin{pmatrix}
I_{2}& 0&       &\\
\mathfrak{L}_{1}   & {I}_{2}   & 0 &\\
        &\mathfrak{L}_2    & {I}_{2} &\ddots\\
&       &\ddots &\ddots\\
\end{pmatrix},\quad
U=\begin{pmatrix}
\mathfrak{U}_{0}   &\mathfrak{D}_{0}&       &\\
{0}   &\mathfrak{U}_{1}    &\mathfrak{D}_{1}&\\
        &0    &\mathfrak{U}_{2} &\ddots\\
&       &\ddots &\ddots\\
\end{pmatrix}
\end{equation}
in which the entries are of the form 
\begin{equation}\label{eq:ind_L}
\mathfrak{L}_{j}=
\begin{pmatrix}
0&0\\
0&\mathfrak{l}_{j}\\
\end{pmatrix},\quad
\mathfrak{D}_{j}=
\begin{pmatrix}
0&0\\
1&0\\
\end{pmatrix},\quad
\mathfrak{U}_j=
             \begin{pmatrix} 0 & 1\\
              \mathfrak{u}_{j}& 0 \end{pmatrix},\quad
{I}_2=
             \begin{pmatrix} 1 & 0\\
              0& 1 \end{pmatrix}.
\end{equation}
Moreover, the following relations hold true
\begin{equation}\label{eq:LU_ind}
\mathfrak{u}_{j}=-\frac{{{\mathfrak T}}_{2j+2}(0)}{{{\mathfrak
T}}_{2j}(0)}=u_j>0,\quad
\mathfrak{l}_{j+1}=\frac{{c}_{j}}{\mathfrak{u}_{j}}=l_{j+1}>0, \quad j\in\dZ_+,
\end{equation}
where $u_j$ and $l_j$ are given by~\eqref{defLUform}.
\end{prop}
\begin{proof}
It is not so hard to see that the equality 
\[
\mathfrak{L}\mathfrak{U}=
\begin{pmatrix}
 \mathfrak{U}_0& \mathfrak{D}_0  & &\\
& & & \\
\mathfrak{L}_{1}\mathfrak{U}_0   &
\mathfrak{L}_1\mathfrak{D}_0+\mathfrak{U}_1
& \mathfrak{D}_1 &\\
&\mathfrak{L}_2\mathfrak{U}_1  & \mathfrak{L}_2\mathfrak{D}_1+\mathfrak{U}_2 &\ddots\\
& & & \\
&       &\ddots &\ddots\\
\end{pmatrix}
=\begin{pmatrix}
\mathfrak{B}_{0}   &\mathfrak{D}_{0}&       &\\
& & & \\
\mathfrak{C}_{0}   &\mathfrak{B}_1    &\mathfrak{D}_{1}&\\
        &\mathfrak{C}_1    &\mathfrak{B}_{2} &\ddots\\
        & & & \\
&       &\ddots &\ddots\\
\end{pmatrix}
\]
reduces to~\eqref{clLUdec}.
\end{proof}

Now, let us return to the monic Jacobi matrix $J$ associated with the Stieltjes function $S\in{\bf S}_{-\infty}$. Recall~\cite{BM04},~\cite{SpZh95},~\cite{Zh97} that the Darboux 
transformation of $J$ is defined in the following manner
\begin{equation}\label{defDarboux}
 J=LU \rightarrow \widetilde{J}=UL.
\end{equation}
It is easy to check that the Darboux transformation $\widetilde{J}$ is a monic Jacobi matrix. Moreover, $\widetilde{J}$ is associated with the measure
$d\widetilde{\mu}(t)=d\mu(t)$~\cite{BM04} or, equivalently, with the function
\[
 \widetilde{S}(\lambda)=\lambda S(\lambda)+s_0=\lambda S(\lambda)+1=\int_{0}^{\infty}\frac{td\mu(t)}{t-\lambda}.
\]
Besides, from~\eqref{defDarboux} one can see how the corresponding  orthogonal polynomials change under the Darboux transformation.
Namely, in view of~\eqref{clLUdec}, the corresponding eigenvector problem~\eqref{JacEigValPr} can be rewritten as follows 
\begin{equation}\label{3termD}
 UL(Up(\lambda))=\lambda(Up(\lambda))
\end{equation}
and simple calculations show that 
\[
 Up(\lambda)=
 \begin{pmatrix}{P_1(\lambda)-\frac{P_1(0)}{P_0(0)}P_0(\lambda)}\\{P_2(\lambda)-\frac{P_2(0)}{P_1(0)}P_1(\lambda)}\\
 {P_3(\lambda)-\frac{P_3(0)}{P_2(0)}P_2(\lambda)}\\ \vdots\end{pmatrix}.
\]
Although the monic polynomials ${P_j(\lambda)-\frac{P_j(0)}{P_{j-1}(0)}P_{j-1}(\lambda)}$ satisfy the three term recurrence relation~\eqref{3termD},
they cannot be orthogonal since they all vanish at 0 and, therefore, the sequence doesn't obey the proper initial condition. However, introducing
\[
 \widetilde{p}(\lambda)=
 \begin{pmatrix}\widetilde{P}_0(\lambda)\\ \widetilde{P}_1(\lambda)\\
 \widetilde{P}_2(\lambda) \\ \vdots\end{pmatrix}
 =\frac{1}{\lambda}Up(\lambda)
\]
we get that
\[
 \widetilde{J}\widetilde{p}(\lambda)=\lambda\widetilde{p}(\lambda)
\]
and $\widetilde{P}_0(\lambda)=1$. Consequently, $\widetilde{P}_j$ are the monic orthogonal polynomials corresponding to $\widetilde{J}=UL$.
Moreover, it can be easily seen from the orthogonality of $P_j$ and the definition of the polynomials $\widetilde{P}_j$, that 
$\widetilde{P}_j$ are orthogonal with respect to $d\widetilde{\mu}(t)=td\mu(t)$. Actually, this is a special case of the Christoffel
formula~\cite[Section~2.5]{Szego}.

The Darboux transformations can be generalized to the case of generalized Jacobi matrices associated with definitizable
functions~\cite{DD10}. In particular, \cite[Theorem 4.2]{DD10} can be complemented
to the following statement.
\begin{thm}\label{DarbouxTheorem}
 Let $\mathfrak{J}$ be a monic generalized Jacobi matrix associated
with $\sZ$ defined by $\sZ(\lambda)=S(\lambda^2)$, $S\in{\bf S}_{-\infty}$ and let $\sJ=\sL\sU$ be its $LU$-factorization of the form~\eqref{ind_LU}.
Then the matrix $\widetilde{\sJ}=\sU\sL$ is the monic classical Jacobi matrix associated with the function $\widetilde{\sZ}$
given by
\[
 \widetilde{\sZ}(\lambda)=\lambda\sZ(\lambda)+\mathfrak{s}_0=\lambda\sZ(\lambda)=\lambda S(\lambda^2).
\]
Besides, the orthogonal polynomials $\widetilde{\sT}_j$ corresponding to $\widetilde{\sJ}$ can be calculated as follows
\begin{equation}
\begin{split}
 \widetilde{\sT}_{2j}(\lambda)&=P_j(\lambda^2),\quad j\in\dZ_+,\\
 \widetilde{\sT}_{2j+1}(\lambda)&=\lambda\widetilde{P}_j(\lambda^2)=
 \frac{1}{\lambda}\left(P_{j+1}(\lambda^2)-\frac{P_{j+1}(0)}{P_{j}(0)}P_{j}(\lambda^2)\right), \quad j\in\dZ_+,
\end{split}
 \end{equation}
where $P_j$ and $\widetilde{P}_j$ correspond to the monic Jacobi matrices $J$ and $\widetilde{J}$, respectively.
\end{thm}
\begin{proof}
 First, simple computations show that
 \[
 %\begin{split}
  \widetilde{\sJ}=\begin{pmatrix}
  \begin{tabular}{ l | c | r | c}
  $\begin{matrix}
  0& 1 \\
  u_0 & 0
 \end{matrix}$ & $\begin{matrix}
  0& 0 \\
  1 & 0
 \end{matrix}$ & $\begin{matrix}
  0& 0 \\
  0 & 0
 \end{matrix}$& $\dots$ \\
  \hline
  $\begin{matrix}
  0& 0 \\
  0 & 0
 \end{matrix}$  & $\begin{matrix}
  0& 1 \\
  u_1 & 0
 \end{matrix}$  & $\begin{matrix}
  0& 0 \\
  1 & 0
 \end{matrix}$& $\ddots$ \\
  \hline
  $\begin{matrix}
  0& 0 \\
  0 & 0
 \end{matrix}$ & $\begin{matrix}
  0& 0 \\
  0 & 0
 \end{matrix}$ & $\begin{matrix}
  0& 1 \\
  u_2 & 0
 \end{matrix}$& $\ddots$ \\
 \hline
 $\vdots$&$\ddots$&$\ddots$&$\ddots$
\end{tabular}
 \end{pmatrix}
 \begin{pmatrix}
  \begin{tabular}{ l | c | r | c}
  $\begin{matrix}
  1& 0 \\
  0 & 1
 \end{matrix}$ & $\begin{matrix}
  0& 0 \\
  0 & 0
 \end{matrix}$ & $\begin{matrix}
  0& 0 \\
  0 & 0
 \end{matrix}$& $\dots$ \\
  \hline
  $\begin{matrix}
  0& 0 \\
  0 & l_1
 \end{matrix}$  & $\begin{matrix}
  1& 0 \\
  0 & 1
 \end{matrix}$  & $\begin{matrix}
  0& 0 \\
  0 & 0
 \end{matrix}$& $\ddots$ \\
  \hline
  $\begin{matrix}
  0& 0 \\
  0 & 0
 \end{matrix}$ & $\begin{matrix}
  0& 0 \\
  0 & l_2
 \end{matrix}$ & $\begin{matrix}
  1& 0 \\
   0& 1
 \end{matrix}$& $\ddots$ \\
 \hline
 $\vdots$&$\ddots$&$\ddots$&$\ddots$
\end{tabular}
 \end{pmatrix}
 \]
 is the tridiagonal matrix 
 \[
 \begin{pmatrix}
0  & 1 &       &&\\
{u}_{0}   &0    &{1}&&\\
        &{l}_1    &0 &1\\
        &&u_1&0&\ddots\\
&&       &\ddots &\ddots\\
\end{pmatrix}.
 %\end{split}
 \]
Taking into account~\eqref{eq:LU_ind} and Favard's theorem, we see that $\widetilde{\sJ}$ is a monic Jacobi matrix associated with a positive measure. 
Next, let us see what happens with the corresponding polynomials. As in the classical case, it is also natural to
introduce the following polynomials
\begin{equation}\label{OPforDT}
\widetilde{\mathfrak{t}}(\lambda)=
 \begin{pmatrix}\widetilde{\sT}_0(\lambda)\\ \widetilde{\sT}_1(\lambda)\\
 \widetilde{\sT}_2(\lambda) \\ \vdots\end{pmatrix}
 =\frac{1}{\lambda}\sU \mathfrak{t}(\lambda).
\end{equation}
Clearly, these polynomials obey the relations
\[
 \widetilde{\sJ}\mathfrak{t}(\lambda)=\lambda\mathfrak{t}(\lambda)
\]
and $\widetilde{\sT}_0(\lambda)=1$. Hence, $\widetilde{\sT}_j$ are monic orthogonal polynomials. 
It follows from~\eqref{OPforDT} that
\begin{equation}\label{defDTpol}
\begin{split}
 \widetilde{\sT}_{2j}(\lambda)&=\frac{1}{\lambda}\sT_{2j+1}(\lambda),\quad j\in\dZ_+,\\
 \widetilde{\sT}_{2j+1}(\lambda)&=\frac{1}{\lambda}\left(\sT_{2j+2}(\lambda)+u_j\sT_{2j}(\lambda)\right),\quad j\in\dZ_+.
 \end{split}
\end{equation}
Applying~\eqref{indOrt1} to~\eqref{defDTpol}, we get that $\widetilde{\sT}_j$ are orthogonal with respect to the functional $\widetilde{\sS}$
defined in the following way
\[
 \widetilde{\sS}(f(\lambda))=\sS(\lambda f(\lambda)), \quad f\in\cP.
\]
Clearly, $\widetilde{\sS}$ generates the same moments as the function $\widetilde{\sF}$.
This observation completes the proof.
\end{proof}

We have just translated the unwrapping transformation 
\[
 S(\lambda)\mapsto\sZ(\lambda)\mapsto\widetilde(\sZ)(\lambda)=\lambda S(\lambda^2)
\]
into the language of matrices:
\begin{equation}\label{DTmatrix}
 J=J(S)\mapsto\sJ=\sJ(\sZ)\mapsto\widetilde{\sJ}=\widetilde{\sJ}(\widetilde{\sZ}).
\end{equation}
Basically, the essential part of~\eqref{DTmatrix} is a Darboux transformation. So, we will call $\widetilde{\sJ}$
the extended Darboux transformation of $J$. 

It is worth mentioning here that $1$-periodic dressing chains of the extended Darboux transformation of $J$, i.e. the equations of the form
\[
J-qI=\widetilde{\sJ}, \quad q>0, 
\]
can lead to pure singular measures supported on Cantor sets~\cite{BMM82}. These dressing chains are related to the Ferromagnetic Ising Model~\cite{BMM82}.
Some other dressing chains generated by Jacobi matrices can be found in~\cite{DTVZh12},~\cite{SZ}.

We conclude this section with the following.
\begin{rem}
 It is not so hard to see that an extended Darboux transformation can be also defined for the Jacobi matrix associated with $F\in{\bf N}_{-\infty}$
 provided that the corresponding monic orthogonal polynomials satisfy
 \[
  P_j(0)\ne 0, \quad j\in\dZ_+.
 \]
In this case, the matrix $\widetilde{\sJ}=\widetilde{\sJ}(\lambda F(\lambda^2))$ is not necessarily generated by a positive measure. Nevertheless,
the matrix $\widetilde{\sJ}$ can be constructed from the $J$-fraction expansion of $\lambda F(\lambda^2)$.
\end{rem}

\section{The Chihara construction and shifted Darboux transformations}

Here we will give a matrix interpretation of the Chihara construction~\cite{Chih64} (see also~\cite{DTVZh12},~\cite{DVZ12}, where it was used to 
study SDG maps) of solutions of the Carlitz problem related to polynomial mappings~\cite{Car61},~\cite{DW63}. 
For this purpose, consider the following function
\begin{equation}\label{IntRofSFunctionSh}
 S(\lambda)=\int_{a}^{\infty}\frac{d\mu(t)}{t-\lambda},
\end{equation}
where $a>0$. The gap $(0,a)$ allows to do shifted Darboux transformations and, therefore, we can generalize Proposition~\ref{ExLU} to the following statement.
\begin{thm}\label{ExLUshift}
Let $S$ have the representation~\eqref{IntRofSFunctionSh} and let $\mathfrak{J}$ be a monic generalized Jacobi 
matrix associated with $\mathfrak{Z}(\lambda)=S(\lambda^2)$, where $S$ is given by~\eqref{IntRofSFunctionSh}. Then for any $\alpha\in(0,\sqrt{a})$ we have that 
\[
(-1)^j\sT_{2j}(\alpha)=(-1)^jP_j(\alpha^2)> 0,\quad j\in\dZ_+ ,
\]
 and 
$\sJ-\alpha I$ admits the following factorization
\begin{equation}\label{ind_LUshift}
\mathfrak{J}-\alpha I=\mathfrak{L}\mathfrak{U},
\end{equation}
where $\mathfrak{L}$ and $\mathfrak{U}$ are block lower and upper
triangular matrices having the forms
\begin{equation}\label{eq:ind_LUshift}
\mathfrak{L}=\begin{pmatrix}
\mathfrak{E}_{0}& 0&       &\\
\mathfrak{L}_{1}   & \mathfrak{E}_{1}   & 0 &\\
        &\mathfrak{L}_2    & \mathfrak{E}_{2} &\ddots\\
&       &\ddots &\ddots\\
\end{pmatrix},\quad
\sU=\begin{pmatrix}
\mathfrak{U}_{0}   &\mathfrak{D}_{0}&       &\\
{0}   &\mathfrak{U}_{1}    &\mathfrak{D}_{1}&\\
        &0    &\mathfrak{U}_{2} &\ddots\\
&       &\ddots &\ddots\\
\end{pmatrix}
\end{equation}
in which the entries are of the form 
\begin{equation}\label{eq:ind_Lshift}
\mathfrak{L}_{j}=
\begin{pmatrix}
0&0\\
0&\mathfrak{l}_{j}\\
\end{pmatrix},\quad
\mathfrak{D}_{j}=
\begin{pmatrix}
0&0\\
1&0\\
\end{pmatrix},\quad
\mathfrak{U}_j=
             \begin{pmatrix} -\alpha & 1\\
              \mathfrak{u}_{j}& 0 \end{pmatrix},\quad
\mathfrak{E}_j=
             \begin{pmatrix} 1 & 0\\
              -\alpha & 1 \end{pmatrix}.
\end{equation}
Moreover, the following relations hold true
\begin{equation}\label{eq:LU_indshift}
\mathfrak{u}_{j}=-\frac{{{\mathfrak T}}_{2j+2}(\alpha)}{{{\mathfrak
T}}_{2j}(\alpha)}=-\frac{P_{j+1}(\alpha^2)}{P_{j}(\alpha^2)}>0,\quad
\mathfrak{l}_{j+1}=\frac{{c}_{j}}{\mathfrak{u}_{j}}>0, \quad j\in\dZ_+.
\end{equation}
\end{thm}
\begin{proof}
 Due to the structures of $\sJ$, $\sL$, and $\sU$, the relation
 \[
  (\sJ-\alpha I)=\sL\sU
 \]
is equivalent to the classical one
\[
 J-\alpha^2 I=LU,
\]
which can be obtained from~\eqref{clLUdec} via considering the monic Jacobi matrix $J^{(\alpha)}=J-\alpha^2 I$
corresponding to the Stietjes function 
\[
 S^{(\alpha)}(\lambda)=S(\lambda+\alpha^2)=\int_{a-\alpha^2}^{\infty}\frac{d\mu(t)}{t-\lambda}
\]
(for more details see~\cite{BM04}). It remains to observe that 
\[
 P_j(0)=\frac{1}{D_{j-1}}
 \begin{vmatrix}
s_0&s_1&\dots&s_j\\
\hdotsfor{4}\\
s_{j-1}&s_j&\dots& s_{2j-1}\\
1&\alpha^2&\dots&\alpha^{2j}\\
\end{vmatrix}=\frac{(-1)^j}{D_{j-1}}
\begin{vmatrix}
s_1-\alpha^2 s_0&\dots&s_j-\alpha^2s_0\\
\hdotsfor{3}\\
s_j-\alpha^2s_{j-1}&\dots& s_{2j-1}-\alpha^2s_{2j-2}
\end{vmatrix}
\]
and $\det(s_{1+l+k}-\alpha^2s_{l+k})_{l.k=0}^{j-1}$ are positive because $s_{k+1}-\alpha^2s_{k}$ are
the moments of the nontrivial positive measure $(t-\alpha^2)d\mu(t)$ supported in $[a,+\infty)$.
\end{proof}

The next step is to generalize Theorem~\ref{DarbouxTheorem} to the shifted case.
\begin{thm}\label{shiftedDarbouxTheorem} In the settings of Theorem~\ref{ExLUshift} we get that 
the matrix 
\[
\widetilde{\sJ}^{(\alpha)}=\sU\sL+\alpha  I          
 \]
 is the monic classical Jacobi matrix associated with the function $\widetilde{\sF}^{(\alpha)}$
given by
\[
 \widetilde{\sF}^{(\alpha)}(\lambda)=(\lambda-\alpha) S(\lambda^2).
\]
Besides, the orthogonal polynomials $\widetilde{\sT}_j^{(\alpha)}$ corresponding to $\widetilde{\sJ}^{(\alpha)}$ can be calculated as follows
\begin{equation}
\begin{split}
 \widetilde{\sT}_{2j}^{(\alpha)}(\lambda)&=P_j(\lambda^2),\quad j\in\dZ_+,\\
 \widetilde{\sT}_{2j+1}^{(\alpha)}(\lambda)&=(\lambda+{\alpha})\widetilde{P}_j^{(\alpha)}(\lambda^2)=
 \frac{1}{\lambda-\alpha}\left(P_{j+1}(\lambda^2)-\frac{P_{j+1}(\alpha^2)}{P_{j}(\alpha^2)}P_{j}(\lambda^2)\right), \quad j\in\dZ_+,
\end{split}
 \end{equation}
where the polynomials $P_j$ correspond to the monic Jacobi matrix $J$ associated with the Stieltjes function $S$.
\end{thm}
\begin{proof}
 First, we can check by straightforward computations that the matrix
 \[
 %\begin{split}
  \widetilde{\sJ}^{(\alpha)}=\begin{pmatrix}
  \begin{tabular}{ l | c | r | c}
  $\begin{matrix}
  -\alpha & 1 \\
\mathfrak{u}_0 & 0
 \end{matrix}$ & $\begin{matrix}
  0& 0 \\
  1 & 0
 \end{matrix}$ & $\begin{matrix}
  0& 0 \\
  0 & 0
 \end{matrix}$& $\dots$ \\
  \hline
  $\begin{matrix}
  0& 0 \\
  0 & 0
 \end{matrix}$  & $\begin{matrix}
  -\alpha& 1 \\
  \mathfrak{u}_1 & 0
 \end{matrix}$  & $\begin{matrix}
  0& 0 \\
  1 & 0
 \end{matrix}$& $\ddots$ \\
  \hline
  $\begin{matrix}
  0& 0 \\
  0 & 0
 \end{matrix}$ & $\begin{matrix}
  0& 0 \\
  0 & 0
 \end{matrix}$ & $\begin{matrix}
  -\alpha& 1 \\
  \mathfrak{u}_2 & 0
 \end{matrix}$& $\ddots$ \\
 \hline
 $\vdots$&$\ddots$&$\ddots$&$\ddots$
\end{tabular}
 \end{pmatrix}
 \begin{pmatrix}
  \begin{tabular}{ l | c | r | c}
  $\begin{matrix}
  1& 0 \\
  -\alpha& 1
 \end{matrix}$ & $\begin{matrix}
  0& 0 \\
  0 & 0
 \end{matrix}$ & $\begin{matrix}
  0& 0 \\
  0 & 0
 \end{matrix}$& $\dots$ \\
  \hline
  $\begin{matrix}
  0& 0 \\
  0 & \mathfrak{l}_1
 \end{matrix}$  & $\begin{matrix}
  1& 0 \\
  -\alpha & 1
 \end{matrix}$  & $\begin{matrix}
  0& 0 \\
  0 & 0
 \end{matrix}$& $\ddots$ \\
  \hline
  $\begin{matrix}
  0& 0 \\
  0 & 0
 \end{matrix}$ & $\begin{matrix}
  0& 0 \\
  0 & \mathfrak{l}_2
 \end{matrix}$ & $\begin{matrix}
  1& 0 \\
   -\alpha& 1
 \end{matrix}$& $\ddots$ \\
 \hline
 $\vdots$&$\ddots$&$\ddots$&$\ddots$
\end{tabular}
 \end{pmatrix}+\alpha I
\]
coincides with the monic classical Jacobi matrix 
\[
 \begin{pmatrix}
-\alpha & 1 &       &&\\
\mathfrak{u}_{0}   &\alpha   &{1}&&\\
        &\mathfrak{l}_1    &-\alpha &1\\
        &&\mathfrak{u}_1&\alpha&\ddots\\
&&       &\ddots &\ddots\\
\end{pmatrix}.
% \end{split}
 \]
Regarding the corresponding polynomials, we have the chain of transformations 
 \[
  (\sJ-\alpha I)\mathfrak{t}(\lambda)=(\lambda-\alpha)\mathfrak{t}(\lambda) \Rightarrow
   (\sU\sL+\alpha I)\sU\mathfrak{t}(\lambda)=\lambda\sU\mathfrak{t}(\lambda)\Rightarrow
   \widetilde{\sJ}^{(\alpha)} \sU\mathfrak{t}(\lambda)=\lambda\sU\mathfrak{t}(\lambda).
 \]
In this case, we have to normalize the vector $\sU\mathfrak{t}$ in the following way 
 \[
 \widetilde{\mathfrak{t}}(\lambda)=
 \begin{pmatrix}\widetilde{\sT}_0^{(\alpha)}(\lambda)\\ \widetilde{\sT}_1^{(\alpha)}(\lambda)\\
 \widetilde{\sT}_2^{(\alpha)}(\lambda) \\ \vdots\end{pmatrix}
 =\frac{1}{\lambda-\alpha}\sU \mathfrak{t}(\lambda)
 \]
 in order the polynomials $\widetilde{\sT}_j^{(\alpha)}$ satisfy the proper initial condition $\widetilde{\sT}_0^{(\alpha)}(\lambda)=1$.
 The orthogonality with respect to the functional %$\widetilde{\sS}^{(\alpha)}$:
\[
 \widetilde{\sS}^{(\alpha)}(f(\lambda))=\sS((\lambda-\alpha) f(\lambda)), \quad f\in\cP,
\]
is also a simple consequence of~\eqref{indOrt1}. For example, let us check that $\widetilde{\sT}^{(\alpha)}_{2j+1}(\lambda)$
is orthogonal to $1$, $\lambda$, \dots, $\lambda^{2j}$:
\[
\begin{split}
 \widetilde{\sS}\left(\lambda^k\widetilde{\sT}^{(\alpha)}_{2j+1}(\lambda)\right)&=
 \sS\left(\lambda^k\left(P_{j+1}(\lambda^2)-\frac{P_{j+1}(\alpha^2)}{P_{j}(\alpha^2)}P_{j}(\lambda^2)\right)\right)\\
  &=\sS\left(\lambda^kP_{j+1}(\lambda^2)\right)+\mathfrak{u}_j\sS\left(\lambda^kP_{j}(\lambda^2)\right)=0, \quad k=0,\dots, 2j.
  \end{split}
\]
Finally, the moments of $ \widetilde{\sS}^{(\alpha)}$ coincide with the moments of $(\lambda-\alpha)S(\lambda^2)$ given
by the asymptotic expansion at infinity. 
\end{proof}
\begin{rem}
 Interestingly, in Theorem~\ref{ExLUshift} and in Theorem~\ref{shiftedDarbouxTheorem} one can take $i\alpha$, $\alpha\in\dR$,
 instead of $\alpha\in{(0,\sqrt{a})}$ and all conclusions will remain the same. In particular, we have that 
\[
\mathfrak{u}_{j}=-\frac{P_{j+1}(-\alpha^2)}{P_{j}(-\alpha^2)}>0,\quad
\mathfrak{l}_{j+1}=\frac{{c}_{j}}{\mathfrak{u}_{j}}>0, \quad j\in\dZ_+.
\]
 Thus, the shifted extended Darboux transformation
 leads to the complex Jacobi matrix
 \[
  \begin{pmatrix}
-i\alpha & 1 &       &&\\
\mathfrak{u}_{0}   &i\alpha   &{1}&&\\
        &\mathfrak{l}_1    &-i\alpha &1\\
        &&\mathfrak{u}_1&i\alpha&\ddots\\
&&       &\ddots &\ddots\\
\end{pmatrix}
 \]
corresponding to the complex measure $\frac{t+i\alpha}{|t|}d\mu(t^2)$ on $\dR$ (cf.~\cite{Chih64}).
\end{rem}

\noindent{\bf Acknowledgments}.
I am deeply indebted to Professor A.S.~Zhedanov for bringing~\cite{Chih64} to my attention and
numerous discussions which improved my understanding of the matter. I am also grateful to
Professor A.I.~Aptekarev for his remarks after my talk on SDG maps at "Quatri\`emes Journ\'ees Approximation"
in Lille. Those remarks stimulated me to write the present paper.

\end{document}